\documentclass{amsart}
\title{Dynamical aspects of foliations with ample normal bundle}

\author{Masanori Adachi}
\address[M. Adachi]{Department of Mathematics, Faculty of Science, Shizuoka University.  836 Ohya, Suruga-ku, Shizuoka 422-8529, Japan.}
\email{adachi.masanori@shizuoka.ac.jp}

\author{Judith Brinkschulte}
\address[J. Brinkschulte]{Universit\"at Leipzig, Mathematisches Institut, PF 100920, D-04009 Leipzig, Germany}
\email{brinkschulte@math.uni-leipzig.de}

\date{November 14, 2021}
\thanks{The first author is partially supported by JSPS KAKENHI Grant Numbers JP18K13422 and JP21H00980.}
\dedicatory{Dedicated to Professor Takeo Ohsawa on the occasion of his 70th birthday}

\usepackage{amsmath,amsthm,amssymb,latexsym}
\usepackage[abbrev]{amsrefs}
\usepackage[all]{xy}
\usepackage{xcolor}
\usepackage{hyperref}

\newtheorem*{MainTheorem}{Main Theorem}
\newtheorem{Theorem}{Theorem}
\newtheorem{Claim}[Theorem]{Claim}

\newtheorem{Proposition}[Theorem]{Proposition}
\newtheorem{Lemma}[Theorem]{Lemma}

\theoremstyle{definition}
\newtheorem{Definition}[Theorem]{Definition}

\theoremstyle{remark}
\newtheorem{Remark}[Theorem]{Remark}
\newtheorem{Example}[Theorem]{Example}

\newcommand\C{\mathbb{C}}  
\newcommand\R{\mathbb{R}}

\newcommand\Z{\mathbb{Z}}

\newcommand\PP{\mathbb{P}}
\newcommand\Ker{\operatorname{Ker}}
\renewcommand\Re{\operatorname{Re}}
\renewcommand\Im{\operatorname{Im}}
\newcommand\Sing{\operatorname{Sing}}
\newcommand\supp{\operatorname{supp}}

\newcommand{\pa}{\partial}
\newcommand{\opa}{\overline\pa}
\newcommand{\ol}{\overline }

\begin{document}

\maketitle

\begin{abstract}
	We prove the following result that was conjectured by Brunella: Let $X$ be a compact complex manifold of dimension $\geq 3$.
	Let $\mathcal{F}$ be a codimension one holomorphic foliation on $X$ with ample normal bundle. 
	Then every leaf of $\mathcal{F}$ accumulates to the singular set of $\mathcal{F}$.
\end{abstract}

\section{Introduction}

Let $X$ be a compact complex manifold of dimension at least two, and let $\mathcal{F}$ be a (singular) holomorphic foliation on $X$. A natural question is whether or not every leaf of $\mathcal{F}$ accumulates to the singular set $\Sing(\mathcal{F})$. An answer to this question will depend on the properties of $X$ and/or $\mathcal{F}$.

\bigskip
It is known that the normal bundle $N_\mathcal{F}$ of the foliation reflects some dynamical properties of $\mathcal{F}$. In particular, as a consequence of Baum--Bott theory, $\Sing(\mathcal{F})$  is nonempty if $\mathcal{F}$ has ample normal bundle. One may therefore try to answer the above question under the assumption that $N_\mathcal{F}$ is ample.

\bigskip
Our main result is as follows.

\begin{MainTheorem}
	Let $X$ be a compact complex manifold of dimension $\geq 3$.
	Let $\mathcal{F}$ be a codimension one holomorphic foliation on $X$ with ample normal bundle $N_\mathcal{F}$. 
	Then every leaf of $\mathcal{F}$ accumulates to $\Sing(\mathcal{F})$.
\end{MainTheorem}

This result was conjectured by Brunella in \cite{brunella1}*{Conjecture 1.1}.

\bigskip
In the special case of $X=\C\PP^n,\ n\geq 3$, the result goes back to Lins Neto \cite{linsneto}; note that in this situation, the normal bundle $N_\mathcal{F}$ is automatically ample since $\C\PP^n$ has positive holomorphic bisectional curvature. In \cite{brunella2}, Brunella gave an affirmative answer to his conjecture when $X$ is a complex torus or, more generally, a compact homogeneous manifold (cf. \cite{correa} for higher codimensional foliations). Also, under the assumption that $\mathrm{Pic}(X)= \Z$, Brunella and Perrone confirmed the conjecture in \cite{brunella-perrone}.

\bigskip
If $\dim X =2$, the problem becomes more difficult, and even for the special case $X=\C\PP^2$, no answer is known.

\bigskip
On the other hand, the Main Theorem might be seen as a further generalization of nonexistence theorems for compact Levi-flat real hypersurfaces that have attracted a great interest in the field of complex analysis over the last decades. Several results concerning the nonexistence of smooth real hypersurfaces invariant by a holomorphic foliation or, more generally, a Levi-flat real hypersurface, related to positivity of the normal bundle can be found in \cite{linsneto}, \cite{siu}, \cite{ohsawa-leviflat1}, \cite{brunella1}, \cite{ohsawa-leviflat2}, \cite{biard-iordan} and \cite{brinkschulte}. In this setting, however, the assumption $\dim X \geq 3$ is crucial: examples for $\dim X =2$ can be found in \cite{brunella1}*{Example 4.2}.

More details will be given in \S\ref{sect:preliminaries}.

\bigskip
We shall explain our main ideas by giving a sketch of the proof. 
We argue by contradiction. Assume that there exists a leaf $\mathcal{L}$ whose closure 
$M := \overline{\mathcal{L}}$ does not intersect $\Sing(\mathcal{F})$.
From Brunella's convexity result (Theorem \ref{thm:brunella}), we know that $X \setminus M$ is strongly pseudoconvex.
Since $\dim X \geq 3$, the codimension two components of $\Sing({\mathcal F})$ are contained in the maximal compact analytic set $A \subset X \setminus M$.

The first ingredient of the proof is the Baum--Bott theory (cf. \cite{brunella-perrone}*{\S2}),
although we use it in a different way from the strategy of Brunella sketched in \cite{brunella1}*{\S4}.
In previous approaches \cite{brunella2,brunella-perrone,correa}, Baum--Bott's formula was used to localize the square of the first Chern class $c^2_1(N_{\mathcal F})$ to $\Sing(\mathcal{F})$.
Instead, we use the vanishing of  the Baum--Bott class over a neighborhood of $M$ to construct a holomorphic connection $\nabla_{\rm hol}$ of $N_{\mathcal F}$ over $X \setminus A$ (\S\ref{sect:step1}).

Using this connection $\nabla_{\rm hol}$, we would like to localize the first Chern class $c_1(N_{\mathcal F})$, not its square, to $A$. 
If $\nabla_{\rm hol}$ was integrable, the desired localization would follow from the residue formula of integrable connections as  in \cite{canales,adachi-biard}.
Due to the lack of a residue formula ready-to-use, we accomplish this localization via that of the first Atiyah class $a_1(N_\mathcal{F})$, inspired by \cite{abate}. 
By rounding $\nabla_{\rm hol}$ around $A$ with a Chern connection, we readily see that $a_1(N_\mathcal{F})$ is localized to a first Atiyah form supported in a small neighborhood $W$ of $A$, which can be chosen to have smooth strongly pseudoconvex boundary.

The main technical step in our proof is the comparison between $c_1(N_\mathcal{F})$ and $a_1(N_\mathcal{F})$.
We take the harmonic projection of the localized first Atiyah form with respect to a complete K\"ahler metric on $W$, and show that the zero extension of this harmonic form gives a desired localization of $c_1(N_\mathcal{F})$.
For this delicate analysis, we exploit  $L^2$ Hodge theory on complete K\"ahler manifolds,  inspired by \cite{ohsawa-hodge1, ohsawa-hodge2} (\S\ref{sect:step2}). 

Once we can localize $c_1(N_\mathcal{F})$ to $A$,  a contradiction easily follows from the $\pa\opa$-lemma and 
 the maximum principle for strictly plurisubharmonic functions in the same way as in \cite{canales,adachi-biard} (\S\ref{sect:step3}).

\subsection*{Notations and conventions}
We will use the notation $|f| \lesssim |g|$ when there exists some positive constant $C$ such that $|f| \leq C |g|$.
We write $|f| \sim |g|$ when both $|f| \lesssim |g|$ and $|g| \lesssim |f|$ hold.
Smoothness means $C^\infty$-smoothness unless otherwise stated.
We use the same notations for holomorphic line bundles and the sheaves of germs of their holomorphic sections.

\section{Preliminaries} 
\label{sect:preliminaries}
For the readers' convenience, we explain some basic definitions and results on holomorphic foliations, the Atiyah class of holomorphic line bundles, and $L^2$ theory for the $\opa$-operator on complete K\"ahler manifolds.

\subsection{Holomorphic foliations}
Let $X$ be a complex manifold of dimension $n \geq 2$. 
In this paper, we discuss foliations in the following sense: 
\begin{Definition}
\label{def:foliation}
We say that a collection of holomorphic 1-forms $\mathcal{F} = \{ \omega_\mu \}$, where $\omega_\mu \in \Omega^1(U_\mu)$ and $\mathcal{U} = \{ U_\mu \}$ is an open covering of $X$, define 
a \emph{codimension one holomorphic foliation} on $X$ if they satisfy the following conditions: for any $\mu$ and $\nu$, 
\begin{enumerate}
\item There exists $g_{\mu\nu} \in \mathcal{O}^*(U_\mu \cap U_\nu)$ such that $\omega_\mu = g_{\mu\nu} \omega_\nu$ on $U_\mu \cap U_\nu$;
\item The analytic set $\{ p \in U_\mu \mid \omega_\mu(p) = 0\}$ has codimension $\geq 2$;
\item The integrability condition is fulfilled: $\omega_\mu \wedge d\omega_\mu = 0$ on $U_\mu$.
\end{enumerate}
\end{Definition}
\bigskip
Here  $\Omega^1$ denotes the sheaf of germs of holomorphic $1$-forms on $X$.
From the first condition, we see that $\{ g_{\mu\nu} \}$ enjoys the cocycle condition and defines a holomorphic line bundle over $X$. We call it the \emph{normal bundle} of $\mathcal{F}$ and denote it by  $N_{\mathcal F}$.
The dual bundle of $N_{\mathcal F}$ is called the \emph{conormal bundle} and denoted by $N^*_{\mathcal F}$. Note that $\{\omega_\mu\}$ defines a global section of $N^*_{\mathcal F}$.

From the first and second condition, the zero sets of the $\omega_\mu$'s glue together and define an analytic set of codimension $\geq 2$ on $X$. We call it the \emph{singular set} of $\mathcal{F}$ and denote it by $\Sing(\mathcal{F})$. 
We also denote $X^\circ := X \setminus \Sing(\mathcal{F})$ for a given foliation $\mathcal{F}$. 

On $X^\circ$, the kernels of the $\omega_\mu$'s define a holomorphic subbundle of $T^{1,0}_{X^\circ}$ of corank one. We call it the \emph{tangent bundle} of $\mathcal{F}$ and denote it by $T_\mathcal{F}$. 
The integrability condition implies that $T_\mathcal{F}$ is integrable in the sense of Frobenius, hence we can find an integral manifold through any point $p \in X^\circ$, that is, a pair of a connected complex manifold $\mathcal{L}$ of dimension $(n-1)$ and an injective holomorphic immersion $\iota \colon \mathcal{L} \to X^\circ$ such that $p \in \iota(\mathcal{L})$ and $\iota_* \colon T^{1,0}_\mathcal{L} \to T_\mathcal{F} \subset T^{1,0}_X$.
A maximal integral manifold is called a \emph{leaf} of $\mathcal{F}$. 
It follows that $X^\circ$ is decomposed into the union of all the leaves of $\mathcal{F}$. 
Therefore, we may think of a foliation as a higher dimensional analogue of flows on $X$. 

From this perspective, it would be natural to seek for a Poincar\'e--Bendixson type property for foliations.
Sometimes a leaf $\mathcal{L}$ of $\mathcal{F}$  approaches  the singular set, $\ol{\mathcal{L}} \cap \Sing(\mathcal{F}) \neq \varnothing$, and sometimes we have a closed leaf $\mathcal{L}$ of $\mathcal{F}$, which 
could be thought of as a counterpart for a periodic orbit. 
We would like to know whether there is another possibility, namely, an exceptional minimal set defined as below:

\begin{Definition} Let $M$ be a nonempty closed subset of $X$. 
\begin{enumerate}
\item We say $M$ is \emph{$\mathcal{F}$-invariant} if $M \setminus \Sing(\mathcal{F})$ is a union of some leaves of $\mathcal{F}$.
\item An $\mathcal{F}$-invariant subset $M$ is called a \emph{minimal set} of $\mathcal{F}$ if it does not contain any proper $\mathcal{F}$-invariant subset.
\item A minimal set $M$ of $\mathcal{F}$ is said to be \emph{non-trivial} if $M \cap \Sing(\mathcal{F})$ is empty.
\item A non-trivial minimal set $M$ of $\mathcal{F}$ is said to be \emph{exceptional} if $M$ is not a closed leaf of $\mathcal{F}$ nor a connected component of $X$.
\end{enumerate}
\end{Definition}

It is a well-known fact that the closure of any leaf is $\mathcal{F}$-invariant and 
contains a minimal set by Zorn's lemma. 
In an exceptional minimal set $M$, any leaf contained in $M$ is dense.

When $X$ is compact and $N_\mathcal{F}$ is ample,
 Baum--Bott theory (cf. \cite{brunella-perrone}*{\S2}) tells us that $\Sing(\mathcal{F})$ cannot be empty.
We can easily see that $\mathcal{F}$ does not have any compact leaf from a folklore argument explained below. 
Hence, nontrivial minimal sets of $\mathcal{F}$, if they exist, must be exceptional in this setting.
\begin{Proposition}
\label{prop:special1}
Let $X$ be a complex manifold of dimension $\geq 2$ and $\mathcal{F}$ a codimension one holomorphic foliation on $X$.
Assume that $N_\mathcal{F}$ admits a smooth Hermitian metric $h$ of positive curvature. Then  $\mathcal{F}$ does not have any compact leaf.
\end{Proposition}
\begin{proof}[Sketch of the proof]
Suppose that there exists a compact leaf $\mathcal{L}$ of $\mathcal{F}$. 
Using the description of a nonsingular holomorphic foliation by foliated charts, 
we obtain a system of local trivializations of $N_\mathcal{F}|_{\mathcal{L}}$ whose transition functions are locally constant.
Since $N_\mathcal{F}|_{\mathcal{L}}$ admits a smooth Hermitian metric of positive curvature, $\mathcal{L}$ is K\"ahler.
Then we can construct a smooth flat Hermitian metric $h_0$ of $N_\mathcal{F}|_{\mathcal{L}}$ (see, for instance, \cite{ueda}*{Proposition 1}).
This metric $h_0$ combined with the given $h$ yields a strictly plurisubharmonic function $\phi := -\log h/h_0$ on $\mathcal{L}$.
This contradicts  the maximum principle.
\end{proof}
In \S\ref{sect:proof}, under the assumption that 
$X$ is compact, $N_\mathcal{F}$ is ample, and $\dim X \geq 3$, 
we will prove the nonexistence of a nontrivial minimal set, that is, an exceptional minimal set by extending this folklore argument substantially.  
One of the key tools is the following convexity result of Brunella. 
\begin{Theorem}[Brunella \cite{brunella2}*{Proposition 3.1}]
\label{thm:brunella}
Let $X$ be a connected compact complex manifold of dimension $\geq 2$ and $\mathcal{F}$ a codimension one holomorphic foliation on $X$.
Let $M \subset X$ be an  $\mathcal{F}$-invariant subset which is disjoint from $\Sing(\mathcal{F})$. 
Assume that $N_\mathcal{F}$ admits a smooth Hermitian metric with positive curvature on a neighborhood of $M$.
Then $X \setminus M$ is strongly pseudoconvex, namely it  admits a smooth plurisubharmonic exhaustion function
$\Psi \colon X \setminus M \to \R$ which is strictly plurisubharmonic on $W \setminus M$ for some neighborhood $W$ of $M$.
\end{Theorem}

When an exceptional minimal set $M$ happens to be a smooth real hypersurface, it must be a \emph{Levi-flat hypersurface}, that is, 
a smooth real hypersurface having a nonsingular smooth foliation by complex hypersurfaces. 
This foliation is called the \emph{Levi foliation} of $M$.
For Levi-flat hypersurfaces, a variant of Brunella's convexity result holds:
\begin{Theorem}[cf. \cite{brunella1}*{Proposition 2.1},\cite{ohsawa-leviflat2}*{Proposition 1.1}, \cite{biard-iordan}*{Proposition 1}]
Let $X$ be a connected compact complex manifold of dimension $\geq 2$ and $M$ a $C^3$-smooth closed Levi-flat hypersurface.
Assume that the holomorphic normal bundle of $M$ admits a $C^2$-smooth Hermitian metric with positive curvature along the leaves of the Levi foliation.
Then $X \setminus M$ is strongly pseudoconvex.
\end{Theorem}

This convexity result played a fundamental role in the proof of the following theorem, which implies that 
an exceptional minimal set of a codimension one holomorphic foliation $\mathcal{F}$ cannot be a smooth compact real hypersurface when $\dim X \geq 3$ and $N_\mathcal{F}$ admits a smooth Hermitian metric of positive curvature.
\begin{Theorem}[\cite{brinkschulte}*{Theorem 1.1}]  \label{leviflat}
Let $X$ be a complex manifold of dimension $\geq 3$. Then there does not exist a smooth compact Levi-flat  hypersurface $M$ in $X$ such that the holomorphic normal bundle of $M$ admits a smooth Hermitian metric with positive curvature along the leaves of the Levi foliation. 
\end{Theorem}

Theorem  \ref{leviflat} generalizes nonexistence theorems for Levi-flat hypersurfaces previously obtained in \cite{linsneto}, \cite{siu}, \cite{ohsawa-leviflat1}, \cite{brunella1}, \cite{ohsawa-leviflat2} and \cite{biard-iordan}.

When we show the nonexistence of compact leaves (Proposition \ref{prop:special1}), or Levi-flat hypersurfaces (Theorem \ref{leviflat}), 
we may take advantage of the property that the normal bundle $N_{\mathcal F}$ is topologically trivial 
along these minimal sets.
However, for a general exceptional minimal set, we do not know whether such a property holds a priori (cf. \cite{ohsawa-reduction}*{Theorem 1.1}).
We shall overcome this difficulty by using the localization of the first Atiyah class of $N_{\mathcal F}$ 
and exploiting the $L^2$ theory for $\opa$ on complete K\"ahler manifolds,
in particular, the $L^2$ Hodge theory, even further.

\subsection{Atiyah class}
Let $X$ be a complex manifold and $L$ a holomorphic line bundle over $X$, and let $\nabla$ be a connection on $L$. 

\begin{Definition}
$\nabla$ is \emph{of type $(1,0)$} (or a \emph{$(1,0)$-connection}) if the connection forms with respect to a holomorphic frame are of type $(1,0)$.
\end{Definition}

\begin{Example} \label{ex:connection}
We will use two kinds of $(1,0)$-connections later.
\begin{enumerate}
\item A holomorphic connection is a special case of a $(1,0)$-connection. In this case, the connection forms with respect to a holomorphic frame are holomorphic $1$-forms.
\item If $L$ moreover has the structure of a Hermitian holomorphic line bundle, then its Chern connection is a $(1,0)$-connection.
\end{enumerate}
\end{Example}

Note that a $(1,0)$-connection is equivalent to the following data: Let $\{ g_{\mu\nu} \} \in H^1(\mathcal{U}, \mathcal{O}^*)$ be a cocycle defining the line bundle $L$ with respect to an open covering $\mathcal{U} = \{ U_\mu \}$ of $X$. Then the connection forms $\{ \eta_\mu \}$ satisfy the gauge transformation law 
\begin{equation} \label{gauge}
\eta_\mu = \frac{d g_{\mu\nu}}{g_{\mu\nu}} + \eta_\nu\quad\text{in $U_\mu\cap U_\nu$}.
\end{equation}

If $\nabla$ is a $(1,0)$-connection, then its curvature form is given by $d \eta_\mu = \pa \eta_\mu + \opa \eta_\mu$. 
This glues together to a well-defined 2-form thanks to (\ref{gauge}). Since the $g_{\mu\nu}$ are holomorphic, also $\opa \eta_\mu$ glues together to a global $(1,1)$-form on $X$.
This form is called {\it the first Atiyah form of $\nabla$} and denoted by $a_1 (\nabla)$.

\begin{Example} We can easily describe the first Atiyah forms for Example \ref{ex:connection}.
\begin{enumerate}
\item If $\nabla$ is a holomorphic connection, then $a_1(\nabla) = 0$. 
\item	If $\nabla$ is the Chern connection of a Hermitian holomorphic line bundle, then $a_1(\nabla)$ is nothing but the first Chern form $c_1(\nabla)$.
\end{enumerate}
\end{Example}

Now suppose we have two $(1,0)$-connections $\nabla_1,\nabla_2$ on $L$, with connection forms $\{\eta_\mu^1\}, \{\eta_\mu^2\}$. Then (\ref{gauge}) implies
$$ \eta_\mu^1 -  \eta_\mu^2 =  \eta_\nu^1 -  \eta_\nu^2\quad\text{in $U_\mu\cap U_\nu$},$$
hence $a_1(\nabla_1,\nabla_2):=  \eta_\mu^1 -  \eta_\mu^2$ defines a global $(1,0)$-form satisfying
\[
\opa a_1(\nabla_1,\nabla_2) = a_1(\nabla_1)-a_1(\nabla_2).
\]
In particular, we see that \emph{the first Atiyah class} of $L$, defined as the class represented by $a_1(\nabla)$ in $H^{1,1}(X)$ for an arbitrary $(1,0)$-connection on $L$, is well defined.
We denote the first Atiyah class of $L$ by $a_1(L)$.

\bigskip
Here we discussed the first Atiyah class only.  Atiyah classes of higher degrees can be defined for holomorphic vector bundles of arbitrary rank. 
We refer the reader to \cite{abate} for more details.

\subsection{$L^2$ theory for $\opa$ on complete K\"ahler manifolds}
\label{sect:L2}
Let $X$ be a complex manifold of dimension $n$ endowed with a smooth Hermitian metric $\omega$. By $L^2_{p,q}(X,\omega)$ (resp. $L^2_k (X,\omega)$) we denote the space of $(p,q)$-forms on $X$ (resp. $k$-forms on $X$) that are integrable with respect to the global $L^2$-norm
$$\Vert u\Vert^2 = \int_X \vert u(x)\vert^2_\omega dV_\omega,$$
where $\vert u(x)\vert_\omega$ is the pointwise Hermitian norm and $dV_\omega$ is the volume form. 

On $X$ we consider the differential operators $d,\pa,\opa$, as well as the corresponding Laplace operators (depending on the metric $\omega$)
$$\Delta = d d^\ast + d^\ast d,\quad \Delta^\prime = \pa\pa^\ast + \pa^\ast\pa,\quad \Delta^{\prime\prime} = \opa\opa^\ast + \opa^\ast\opa .$$
All these operators extend naturally as linear, closed, densely defined operators on the previously defined $L^2$-spaces in the sense of distributions.
We will then consider the spaces of harmonic forms
$$\mathcal{H}^k(X,\omega)= \lbrace u\in L^2_k(X,\omega)\mid \Delta u =0\rbrace,$$
$$\mathcal{H}^{p,q}(X,\omega)= \lbrace u\in L^2_{p,q}(X,\omega)\mid \Delta^{\prime\prime} u =0\rbrace.$$

Of special importance will be the case of a {\it complete} metric, which is equivalent to the existence of a smooth exhaustion function $\psi\colon X\longrightarrow\R$ satisfying $\vert d\psi\vert_\omega \leq 1$. 
In this case, the spaces $\mathcal{D}^{p,q}(X)$ of smooth compactly supported forms are dense in $L^2_{p,q}(X,\omega)\cap\mathrm{Dom}\opa\cap\mathrm{Dom}\opa^\ast$ in the norm $\Vert u\Vert + \Vert\opa u\Vert + \Vert\opa^\ast u\Vert$, and we have
$$\mathcal{H}^k(X,\omega)=  \lbrace u\in L^2_k(X,\omega)\mid d u = d^\ast u =0\rbrace,$$
$$\mathcal{H}^{p,q}(X,\omega)= \lbrace u\in L^2_{p,q}(X,\omega)\mid \opa u =  \opa^\ast u = 0\rbrace.$$

We will also need the following $L^2$ de Rham and Dolbeault cohomology groups
$$H^k_{L^2}(X,\omega) = L^2_k(X,\omega)\cap \Ker d / L^2_k(X,\omega)\cap \Im d,$$
$$H^{p,q}_{L^2}(X,\omega) = L^2_{p,q}(X,\omega)\cap \Ker \opa / L^2_{p,q}(X,\omega)\cap \Im \opa.$$
Then, if $\omega$ is complete, we have a natural isomorphism
$$\mathcal{H}^{p,q}(X,\omega) \simeq H^{p,q}_{L^2}(X,\omega),$$
provided the latter cohomology group is finite dimensional, as well as the corresponding isomorphism for the $L^2$ de Rham cohomology. 

If moreover the metric $\omega$ is {\it K\"ahler}, we have $\Delta = 2\Delta^\prime = 2\Delta^{\prime\prime}$, which implies the usual Hodge decomposition
$$H^k_{L^2}(X,\omega) {\simeq} \bigoplus_{p+q=k} H^{p,q}_{L^2}(X,\omega)$$
if the $L^2$ de Rham and Dolbeault cohomology groups are finite dimensional.

\bigskip
In the following sections, we will also need weighted $L^2$-spaces of the type $L^2_{p,q}(X,\omega,\varphi)$, consisting of forms that are integrable with respect to the weighted $L^2$-norm
$$\Vert u\Vert_\varphi^2 = \int_X \vert u(x)\vert^2_\omega e^{-\varphi}dV_\omega,$$
where $\varphi: X\longrightarrow\R$ is a (smooth) weight function. In this case,  the adjoint $\opa^\ast_\varphi$ depends not only on the metric $\omega$, but also on the weight function $\varphi$.

\bigskip
We refer the reader to \cite{Demailly-Hodge} for the details on the discussion above.

\section{Proof of Brunella's conjecture} \label{sect:proof}
\begin{proof}[Proof of the Main Theorem]
The proof is by contradiction. Assume that there exists a leaf $\mathcal{L}$ whose closure 
$M := \overline{\mathcal{L}}$ does not intersect  $S := \Sing(\mathcal{F})$.
We may assume that $X$ is connected by replacing it with the connected component containing $M$ if necessary.
From Theorem \ref{thm:brunella}, $X \setminus M$ is strongly pseudoconvex.
Denote by $A$ the maximal compact analytic set of $X \setminus M$,
and by $S^*$ the codimension two part of $S$. 
Note that $S^* \subset A$ since $\dim X \geq 3$.

\subsection{First step} \label{sect:step1}
In this step, we shall construct a holomorphic connection of $N_{\mathcal{F}}$ over $X \setminus A$.

Cover $X \setminus S^*$ by open sets $\{U_\mu \}$ and take the collection of holomorphic 1-forms $\{ \omega_\mu \}$ defining $\mathcal{F}$ on $X \setminus S^*$.
By taking a refinement of $\{U_\mu\}$ if necessary, the integrability condition yields $\beta_\mu \in \Omega^1(U_\mu)$ such that $d\omega_\mu = \beta_\mu \wedge \omega_\mu$ thanks to Malgrange's theorem  \cite{malgrange}*{Th\'eor\`eme (0.1)}.
We denote the cocycle defining the normal bundle $N_\mathcal{F}$ by  $\{g_{\mu\nu}\}$ as in Definition \ref{def:foliation}.
\begin{Claim}
The cocycle $\{ \gamma_{\mu\nu} := d g_{\mu\nu}/g_{\mu\nu} - \beta_\mu + \beta_\nu \}$ defines a cohomology class, 
called \emph{the Baum--Bott class}, $BB_\mathcal{F} \in H^1(X \setminus S^*, N^*_\mathcal{F})$.
\end{Claim}
\begin{proof}
On $U_\mu \cap U_\nu$, we have 
\begin{align*}
\gamma_{\mu\nu} \wedge \omega_\nu 
& = \frac{d g_{\mu\nu}}{g_{\mu\nu}} \wedge \omega_\nu - \beta_\mu \wedge \omega_\nu + \beta_\nu \wedge \omega_\nu\\
&= \frac{1}{g_{\mu\nu}}d(g_{\mu\nu}\omega_\nu)  - d\omega_\nu - \beta_\mu \wedge \omega_\nu +  d\omega_\nu\\
&= \frac{1}{g_{\mu\nu}}d\omega_\mu  - \frac{1}{g_{\mu\nu}}\beta_\mu \wedge \omega_\mu = 0.
\end{align*}
Hence, there exists $f_{\mu\nu} \in \mathcal{O}(U_\mu \cap U_\nu)$ such that $\gamma_{\mu\nu} = f_{\mu\nu} \omega_\nu$
thanks to Riemann's extension theorem.
Since $\{\omega_\mu\}$ defines a global section of $N^*_\mathcal{F}$, we see that $\gamma_{\mu\nu}$ is a section of $N^*_{\mathcal{F}}$ over $U_\mu \cap U_\nu$.
The cocycle condition is clear. 
\end{proof} 

Take open connected neighborhoods $V$ and $W$ of $A$ so that $V \Subset W \Subset X \setminus M$ and $W$ has smooth strictly pseudoconvex boundary.
\begin{Claim}
$BB_\mathcal{F}|_{X \setminus \overline{W}} = 0$ in $H^1(X \setminus \overline{W}, N^*_{\mathcal{F}})$.
\end{Claim}
\begin{proof}
In this proof, we identify the sheaf cohomology groups with the Dolbeault cohomology groups. 
Let $g$ be a smooth $\opa$-closed $(0,1)$-form on $X \setminus \overline{V}$ with values in $N_\mathcal{F}^\ast$ 
that represents $BB_\mathcal{F}|_{X \setminus \overline{V}}$.
We would like to show that $g$ is $\opa$-exact on $X \setminus \overline{W}$.

Take a compactly supported smooth function $\rho \colon W \to [0,1]$ such that $\rho = 1$ on $\ol{V}$.
Then $\tilde g := (1-\rho)g$ is a smooth $(0,1)$-form on $X$ with values in $N_\mathcal{F}^\ast$,
and $\opa \tilde g$ is a $(0,2)$-form,  compactly supported in $W$.
	Since $W$ is strongly pseudoconvex, $\dim X = n\geq 3$ and $N_{\mathcal{F}}$ is positive, $H^{n,n-2}(W, N_{\mathcal{F}})$ vanishes as a consequence of the vanishing theorem by Grauert--Riemenschneider. But then Serre duality implies that $H^{0,2}_{c}(W, N^*_{\mathcal{F}})$ vanishes. 
	Therefore we can solve the equation $\opa u = \opa\tilde g$ with $u$ compactly supported in $W$,
and obtain a smooth $\opa$-closed extension $\tilde g -u$ of $g|_{X \setminus \overline{W}}$ to $X$.

Since $N_{\mathcal{F}}$ is assumed to be positive,  
$H^1(X, N^*_\mathcal{F})$ vanishes thanks to Kodaira's vanishing theorem. 
It follows that $\tilde g -u$ is $\opa$-exact on $X$, and hence, $g$ is $\opa$-exact on $X \setminus \overline{W}$.
	\end{proof}
By taking a refinement of $\{U_\mu\}$ if necessary, we find $\gamma_\mu \in \Omega^1(U_\mu \setminus \overline{W}) $ 
satisfying $\gamma_{\mu\nu} = \gamma_\mu - \gamma_\nu$ in $(U_\mu \cap U_\nu) \setminus \overline{W}$. 
Then $\hat{\beta}_\mu := \beta_\mu + \gamma_\mu$ satisfies the gauge transformation law $\hat{\beta}_\mu = dg_{\mu\nu}/g_{\mu\nu} + \hat{\beta}_\nu$, hence, $\{\hat{\beta}_\mu\}$ defines a holomorphic connection of  $N_\mathcal{F}|_{X \setminus \overline{W}}$.
By shrinking $V$ and $W$, we obtain a holomorphic connection $\nabla_{\rm hol}$ of $N_\mathcal{F}|_{X \setminus A}$.

\subsection{Second step} \label{sect:step2}
In this step, we shall show that the first Atiyah class $a_1(N_{\mathcal F})$ of $N_{\mathcal{F}}$ is represented by a $d$-closed $L^2$ (1,1)-form which is supported in $\ol{W}$ and smooth in $X \setminus \pa W$. 

We take an arbitrary smooth Hermitian metric $h_0$ of $N_\mathcal{F}$ and consider its Chern connection $\nabla_0$. 
Take a compactly supported smooth function $\rho \colon W \to [0,1]$ such that $\rho = 1$ on $\ol{V}$.
Defining $\nabla := \rho \nabla_0 + (1-\rho)\nabla_{\rm hol}$, we get a smooth $(1,0)$-connection of $N_\mathcal{F}$ which is holomorphic in $X \setminus \supp \rho$.
We choose $a_1(\nabla)$ as a representative of $a_1(N_{\mathcal F})$. 
Since $\nabla$ agrees with the holomorphic connection $\nabla_{\rm hol}$ over $X \setminus \supp \rho$,
$a_1(\nabla) = 0$ in $X \setminus \supp \rho$.
Hence, $a_1(\nabla)$ can be seen as a $\opa$-closed smooth $(1,1)$-form compactly supported in $W$.

\bigskip
We are going to find a $d$-closed $(1,1)$-form supported in $\ol{W}$ which is $\opa$-cohomologues to $a_1(\nabla)$
by employing $L^2$ Hodge theory on $W$. 
We equip $W$ with a complete K\"ahler metric $\omega$ of the form
\[
\omega = \omega_0  + i\pa\opa (-\log \delta),
\] 
where $\omega_0$ is a K\"ahler metric on $X$, which exists since $X$ is projective, 
and $-\delta \colon \ol{W} \to [-1, 0]$ is a smooth plurisubharmonic defining function for $W$ which 
is strictly plurisubharmonic on a neighborhood of $\pa W$.

\begin{Lemma} \label{lem:hodge}
	$H^{1,1}_{L^2}(W,\omega)$ and $\mathcal{H}^{1,1}(W,\omega)$ are finite dimensional and isomorphic if $\dim X = n \geq 3$.
	\end{Lemma}
	
Although this is a classical result (cf. \cite{ohsawa-hodge1,ohsawa-hodge2,ohsawa-takegoshi,demailly}), we shall give a self-contained proof for the readers' convenience.
\begin{proof}
	The idea is to use the twisting trick of Berndtsson and Siu, making use of the advantage that the function $-\log\delta$ satisfies the Donnelly--Fefferman condition outside a compact of $W$.
	Precisely speaking, since $\omega = \omega_0 +i\pa\opa(-\log\delta) = \omega_0 + i \frac{\pa\opa(-\delta)}{\delta}+ i\frac{\pa\delta}{\delta}\wedge\frac{\opa\delta}{\delta}$, we may choose $0< t <\frac{1}{4}$ and a compact $K\subset W$  such that $\varphi := -t\log\delta$ satisfies
	\begin{equation} \label{DF}
		\vert\pa\varphi\vert_\omega^2\leq \frac{t}{{4}}\quad\text{in $W\setminus K$}.
		\end{equation}
Let $(0\leq) \lambda_1\leq\dots\leq\lambda_{n}$ be the eigenvalues of $i\pa\opa\varphi$ with respect to the metric  $\omega$. Then, since $\omega= \omega_0 +\frac{1}{t}i\pa\opa\varphi$, enlarging $K$ if necessary, we may assume that 
\begin{equation}  \label{ev}
	\lambda_j = t + \varepsilon_j \quad\text{in $W\setminus K$}
	\end{equation}
with $\vert\varepsilon_j\vert \leq \frac{t}{2n}$ for all $1\leq j\leq n$.	

\bigskip
From the Bochner--Kodaira--Nakano inequality (see \cite{demailly-book}) it follows that
	\begin{equation*}
		\Vert\opa u\Vert^2_{-\varphi} + \Vert\opa^\ast_{-\varphi }u\Vert^2_{-\varphi} \geq \int_W \big((\lambda_1 +\ldots +\lambda_{n-1}) -\lambda_n\big)\vert u\vert^2_\omega e^\varphi dV_\omega\end{equation*}
	for $u\in \mathcal{D}^{1,1}(W)$.
	Using (\ref{ev}), we have the estimate
	\begin{equation*}
		\lambda_1 +\ldots +\lambda_{n-1} -\lambda_n\geq t(n-2) -\frac{t}{2} \geq \frac{t}{2}\quad \mathrm{in}\ W\setminus K
	\end{equation*}
since $n \geq 3$.
Hence for every $u\in \mathcal{D}^{1,1}(W\setminus K)$ we have the estimate
\begin{equation} \label{basic}
\frac{t}{2} \Vert u\Vert_{-\varphi}^2\leq \Vert \opa u\Vert^2_{-\varphi} + \Vert \opa^\ast_{-\varphi} u \Vert^2_{-\varphi}.
\end{equation}

\bigskip
Now we substitute $u = v e^{-\varphi/2}$. It is not difficult to see that
$$
\vert \opa u\vert^2_\omega  e^{\varphi} 
 = \vert \opa v - \frac{1}{2} \opa \varphi \wedge v \vert^2_\omega  
 \leq 2 \left( \vert\opa v\vert^2_{\omega} + \frac{1}{4} \vert \opa\varphi\vert^2_{\omega} \vert v\vert^2_{\omega} \right) 
 \leq 2 \vert\opa v\vert^2_{\omega} + \frac{t}{8} \vert v\vert^2_{\omega}.
$$
Since $\opa^\ast_{-\varphi} = e^{-\varphi} \opa^\ast e^{\varphi}$, we likewise get
\begin{align*}
\vert\opa^\ast_{-\varphi} u\vert^2_{\omega} e^{\varphi} &= \vert \opa^\ast (e^{\varphi/2} v) \vert^2_{\omega} e^{-\varphi}
= \vert \opa^\ast v  + \frac{1}{2} (\opa \varphi)^* v \vert^2_\omega  \\
&\leq 2 \left(\vert\opa^\ast v\vert^2_{\omega} + \frac{1}{4} \vert \opa\varphi\vert^2_{\omega}  \vert v\vert^2_{\omega} \right) \leq 2\vert\opa^\ast v\vert^2_{\omega} + \frac{t}{8}\vert v\vert^2_{\omega}.
\end{align*}
Together with (\ref{basic}), these two inequalities imply
\begin{equation}  \label{estspace2}
	 \frac{t}{8} \Vert v\Vert^2 \leq( \Vert\opa v\Vert^2 + \Vert\opa^\ast v\Vert^2)
\end{equation}
for all $v\in\mathcal{D}^{1,1}(W\setminus K)$. 

It is now standard to conclude that for any compact $L$ containing $K$ in its interior, there exists a positive constant $C_L$ such that
$$\Vert v\Vert^2 \leq C_L ( \Vert \opa v\Vert^2 + \Vert\opa^\ast v\Vert^2 + \int_L \vert v\vert^2_{\omega} dV_{\omega})$$
for all $v\in\mathcal{D}^{1,1}(W)$, hence for all $v\in L^2_{1,1}(W,\omega)$ by completeness of the metric $\omega$. But the above estimate implies that $H^{1,1}_{L^2}(W,\omega)$ is finite dimensional and isomorphic to  $\mathcal{H}^{1,1}(W,\omega)$. 
	\end{proof}

Since $a_1(\nabla)$ is a $\opa$-closed $(1,1)$-form compactly supported in $W$, Lemma \ref{lem:hodge} yields 
a smooth $(1,1)$-form $\gamma \in \mathcal{H}^{1,1}(W, \omega)$ and $\beta_1 \in L^2_{1,0}(W, \omega)$ 
such that $a_1(\nabla) - \gamma = \opa \beta_1$ on $W$.

Note that 
\[
dV_{\omega} \sim \omega^n  \sim \frac{1}{\delta^{n+1}}\omega_0^n \sim \delta^{-n-1}dV_{\omega_0}.
\]
Since $\delta^2 \omega \lesssim \omega_0$ on $W$, it holds that $\delta^{-2} |\beta|^2_{\omega} \gtrsim |\beta|^2_{\omega_0}$ for any 1-form $\beta$ on $W$,
and $\delta^{-4}|\alpha|^2_{\omega} \gtrsim  |\alpha|^2_{\omega_0}$ for any 2-form $\alpha$ on $W$.
Hence, 
\begin{align*}
\int_W | \beta_1 |^2_{\omega_0} \delta^{-n+1} dV_{\omega_0}
& \lesssim \int_W | \beta_1 |^2_{\omega} \delta^{-n-1} dV_{\omega_0}
 \sim \int_W | \beta_1 |^2_{\omega} dV_{\omega} < \infty,\\
\int_W | \gamma |^2_{\omega_0} \delta^{-n+3} dV_{\omega_0}
& \lesssim \int_W | \gamma |^2_{\omega} \delta^{-n-1} dV_{\omega_0}
 \sim \int_W | \gamma |^2_{\omega} dV_{\omega} < \infty,
\end{align*}
therefore,
\begin{align*}
\beta_1 &\in L^2_{1,0}(W,  \omega_0, (n-1)\log\delta) \subset L^2_{1,0}(W, \omega_0, \log\delta),\\
\gamma &\in L^2_{1,1}(W, \omega_0, (n-3)\log\delta) \subset L^2_{1,1}(W, \omega_0)
\end{align*}
since $\dim X = n \geq 3$. 

\bigskip
We denote the extensions by zero of $\gamma$ and $\beta_1$ to $X$ by $\tilde{\gamma}$ and $\tilde{\beta}_1$ respectively. 
Then $\tilde{\gamma}$ and $\tilde{\beta_1}$ belong to $L^2_{\bullet,\bullet}(X, \omega_0)$.
Moreover, they satisfy the above $\opa$-equation not only on $W$ but also on $X$:
\begin{Claim}
\label{claim:dbar}
$a_1(\nabla) - \tilde{\gamma} = \opa \tilde{\beta}_1$ on $X$.
\end{Claim}

\begin{proof} Setting $\alpha_1 = a_1(\nabla) - \tilde{\gamma}$, we want to show that $\opa\tilde{\beta}_1 = \alpha_1$ in the sense of distributions on $X$. 
	That is, we have to show that
	\begin{equation}  \label{toshow}
		\int_{W} \beta_1\wedge \opa g = \int_{W} \alpha_1\wedge g.
	\end{equation}
	for $g\in \mathcal{C}^\infty_{n-1,n-1}(X)$
	
		Let $\chi\in\mathcal{C}^\infty(\R,\R)$ be a function such that $\chi(t)=0$ for $t \leq \frac{1}{2}$ and $\chi(t)=1$ for $t \geq 1$. Set $\chi_j = \chi(j\delta)\in\mathcal{D}(W)$. 
	Then $\chi_j g \in\mathcal{D}^{n-1,n-1}(W)$, and since $\opa \beta_1 = \alpha_1$ in $W$, we therefore have
	$$\int_{W} \alpha_1\wedge \chi_j g = \int_{W} \beta_1\wedge\opa (\chi_j g) = \int_{W} \beta_1\wedge(\opa \chi_j \wedge g + \chi_j\opa g).$$
	As $\alpha_1$ has $L^2$ coefficients on $W$, the integral of $\alpha_1\wedge\chi_j g$   converges  to the integral of $\alpha_1\wedge g$  as $j$ tends to infinity.  Since $\beta_1\in L^1_{1,0}(X)$, the integral of $\beta_1\wedge\chi_j\opa g$ converges to the integral of $\beta_1\wedge\opa g$. 
	The remaining term can be estimated as follows: using the Cauchy--Schwarz inequality we have
	\begin{equation}  \label{CauchySchwarz}
		\left\vert\int_{W} \beta_1\wedge \opa\chi_j \wedge g \right\vert^2 \leq
		\sup_{W}\vert g\vert^2_{\omega_0} \int_{\lbrace 0<\delta \leq \frac{1}{j}\rbrace} \vert \beta_1\vert^2_{\omega_0}\delta^{-1}dV_{\omega_0} \cdot \int_{W}\vert\opa\chi_j\vert^2_{\omega_0}\delta dV_{\omega_0}.
	\end{equation}
	
	Since $\beta_1\in L^2_{1,0}(W,\omega_0,\log\delta)$, the integral $\int_{\lbrace 0 <\delta \leq \frac{1}{j}\rbrace} \vert \beta_1\vert^2_{\omega_0}\delta^{-1}dV_{\omega_0}$ converges to $0$ when $j$ tends to infinity. 
	We estimate the second integral as follows
	\begin{equation*}
		\int_{W} \vert \opa\chi_j\vert^2_{\omega_0} \delta dV_{\omega_0} 
		 \lesssim  \int_{\lbrace 0<\delta \leq \frac{1}{j}\rbrace} j^2 \delta dV_{\omega_0} \\
		 \lesssim  j \mathrm{Vol}_{\omega_0}(\lbrace 0 <\delta\leq \frac{1}{j}\rbrace) \leq\ \mathrm{cte}.
	\end{equation*}

	Combining this estimate with (\ref{CauchySchwarz}), we have proved that $\int_{W} \beta_1\wedge\opa\chi_j \wedge g$ converges to $0$ when $j$ tends to infinity. Equation (\ref{toshow}) follows.
	
	\end{proof}

Since $\omega$ is K\"ahler, the harmonic form $\gamma$ is $d$-closed on $W$ (see §\ref{sect:L2}). 
The extension by zero, $\tilde{\gamma}$, is also $d$-closed on $X$ and gives the desired representative of $a_1(N_{\mathcal{F}})$ over $X$. 
\begin{Claim} \label{claim:closedness}
The $(1,1)$-form $\tilde{\gamma}$ is $d$-closed on $X$.
\end{Claim}
\begin{proof}
Since $a_1(\nabla) - \tilde{\gamma} = \opa \tilde{\beta}_1$ holds on $X$, $\tilde{\gamma}$ is $\opa$-closed on $X$.
We show that $\pa \tilde{\gamma} = 0$ in the sense of distributions on $X$.
	That is, we have to show that
	\begin{equation}  
		\int_{W} \tilde{\gamma} \wedge \pa g = 0
	\end{equation}
	for $g\in \mathcal{C}^\infty_{n-2,n-1}(X)$

Take $\chi$ and $\chi_j$ as in the proof of Claim \ref{claim:dbar}. Then, 
\[
\int_{W} \tilde{\gamma} \wedge \pa g = 
\lim_{j \to \infty} \int_{W} \chi_j \gamma \wedge \pa g =
\lim_{j \to \infty} \int_{W}  \gamma \wedge \pa (\chi_j g) - \lim_{j \to \infty} \int_{W}  \gamma \wedge \pa \chi_j\wedge g.
\]
Since $\pa\gamma =0$ in $W$ this implies
\begin{equation} \label{dclosed}
\int_{W} \tilde{\gamma} \wedge \pa g = 
-\lim_{j \to \infty} \int_{W}  \gamma \wedge \pa \chi_j\wedge g  .
\end{equation}

\bigskip
To compute (\ref{dclosed}), we need to decompose $\gamma$.
Since $\pa \delta \wedge \opa \delta$ does not vanish on $W \setminus K$ for some compact $K$,
we can decompose the smooth $(1,1)$-form $\gamma$ as 
\[
\gamma = \gamma_1  + \gamma_2 \quad\text{over $W \setminus K$},
\]
where $\gamma_2$ is a smooth multiple of $\pa\delta\wedge\opa\delta$ and 
$\gamma_1$ is orthogonal to $\pa\delta\wedge\opa\delta$ with respect to the metric $\omega_0$.

We would like to show that $\gamma_1$ has better integrability than $\gamma$ with respect to the metric $\omega_0$. 
For this, let us use an orthogonal decomposition of the metric $\omega_0$ in its tangential and normal parts:
\[
\omega_0 = \omega_{\rm t} + \omega_{\rm n}\quad\text{over $W \setminus K$},
\]
where $\omega_{\rm t}$ corresponds to a smooth Hermitian metric of the subbundle $\Ker \pa\delta \subset T^{1,0}_{W \setminus K}$, and $\omega_{\rm n}$ is a smooth positive multiple of $i\pa\delta\wedge\opa\delta$.
Using this decomposition, we rescale $\omega_0$ to another smooth Hermitian metric $\omega^\prime$ on $W \setminus K$
given by 
\[
\omega^\prime := \frac{1}{\delta}\omega_{\rm t} + \frac{1}{\delta^2}\omega_{\rm n}
=  \frac{1}{\delta}\omega_0 + \frac{1-\delta}{\delta^2} \omega_{\rm n}.
\]
Notice that $\gamma_1$ and $\gamma_2 $ are still orthogonal with respect to $\omega^\prime$.
In particular, it holds that $ |\gamma_1 |^2_{\omega'} \leq |\gamma|^2_{\omega'}$. 
Also, we have $\omega \sim \omega^\prime$ on $W \setminus K$, by enlarging $K$ if necessary, since 
\[
\omega \sim i\pa\opa(-\log\delta) = i \frac{\pa\opa(-\delta)}{\delta}+ i\frac{\pa\delta \wedge \opa\delta}{\delta^2}
\sim  \frac{1}{\delta}\omega_0 + \frac{1}{\delta^2} i\pa\delta \wedge \opa\delta.
\]
We thus have $\gamma_1 \in L^2_{1,1}(W \setminus K,\omega^\prime)$.

It follows from the definitions of $\gamma_1$ and $\omega'$ that  $\vert \gamma_1\vert^2_{\omega^\prime}\geq \delta^3 \vert\gamma_1\vert^2_{\omega_0}$.
Then, 
\[
\int_{W \setminus K} | \gamma_1 |^2_{\omega_0} \delta^{-1} dV_{\omega_0}
\leq \int_{W \setminus K} | \gamma_1 |^2_{\omega'} \delta^{-4} dV_{\omega_0}
\lesssim \int_{W \setminus K} | \gamma_1 |^2_{\omega'} dV_{\omega'} < \infty
\]
since $n \geq 3$ and $dV_{\omega^\prime} \sim dV_{\omega} \sim \delta^{-(n+1)}dV_{\omega_0}$. 
Therefore, we have $\gamma_1\in L^2_{1,1}(W \setminus K,\omega_0,\log\delta)$.

\bigskip
We continue with the computation of (\ref{dclosed}).
Since $\pa\chi_j = j\chi^\prime(j\delta)\pa\delta$, we get
\[
 \int_{W}  \gamma \wedge \pa \chi_j\wedge g  =  \int_{W}  \gamma_1 \wedge \pa \chi_j\wedge g 
\]
for $j$ enough large, and, as in the proof of Claim \ref{claim:dbar},
 
 \begin{align*}
 \left\vert \int_{W}  \gamma_1 \wedge \pa \chi_j\wedge g \right\vert \leq \sup_{W}\vert g\vert^2_{\omega_0} \int_{\lbrace 0<\delta \leq \frac{1}{j}\rbrace} \vert \gamma_1 \vert^2_{\omega_0}\delta^{-1}dV_{\omega_0} \cdot \int_{W}\vert\opa\chi_j\vert^2_{\omega_0}\delta dV_{\omega_0}\\
  \lesssim \int_{\lbrace 0<\delta \leq \frac{1}{j}\rbrace} \vert \gamma_1 \vert^2_{\omega_0}\delta^{-1}dV_{\omega_0} \underset{j\rightarrow\infty}{\longrightarrow} 0.
 \end{align*}
\end{proof}

\begin{Remark}
When $\dim X = n \geq 4$, Claim \ref{claim:closedness} follows from the integrability 
$\gamma \in L^2_{1,1}(W, \omega_0, (n-3)\log\delta) \subset L^2_{1,1}(W, \omega_0, \log\delta)$
in the same way as in Claim \ref{claim:dbar}.
The delicate use of the orthogonal decomposition $\gamma = \gamma_1 + \gamma_2$ is required for $n  = 3$.
\end{Remark}

\subsection{Third step} \label{sect:step3}
In this step, we construct a flat Hermitian metric of $N_{\mathcal F}$ over $M$, and deduce a contradiction.

Since both $a_1(\nabla_0)$ and $a_1(\nabla)$ represent $a_1(N_{\mathcal F})$, 
there exists a smooth $(1,0)$-form $\beta_2$ on $X$ such that 
\[
\opa\beta_2 = a_1(\nabla_0) -a_1(\nabla) = a_1(\nabla_0) - (\tilde{\gamma} + \opa \tilde{\beta}_1).
\]
Notice that $a_1(\nabla_0)$ is exactly the first Chern form of $\nabla_0$, a $d$-closed smooth $(1,1)$-form. 
Therefore, 
\[
a_1(\nabla_0) - \tilde{\gamma} = \opa (\tilde{\beta}_1 +  \beta_2).
\]
is a $d$-closed, $\opa$-exact $L^2$ $(1,1)$-form which is smooth on $X \setminus \pa W$.

Since $X$ is projective, hence, K\"ahler, we may apply the $\partial\overline{\partial}$-lemma 
to obtain an $L^2$ function $\psi \colon X \to \C$ smooth on $X \setminus \partial W$ such that 
\[
a_1(\nabla_0) - \tilde{\gamma} = i\partial\overline{\partial} \psi.
\]
Since $\tilde{\gamma}$ has its support in $\ol{W}$, it holds on $X \setminus \ol{W}$ that 
\[
\frac{i}{2\pi} \opa \pa \log h_0 = i\partial\overline{\partial} \psi,
\]
which implies
\[
\frac{i}{2\pi} \opa \pa \log h_0 = i\partial\overline{\partial} \Re \psi
\]
on $X \setminus \ol{W}$.
Hence, $h_1 := h_0 e^{2\pi \Re\psi}$ gives a smooth flat Hermitian metric of $N_{\mathcal{F}}$ over $X \setminus  \overline{W}$.

Since we assume that $N_{\mathcal{F}}$ is ample, there exists a smooth Hermitian metric $h_2$ of $N_\mathcal{F}$ with positive curvature. 
Then, $\phi := -\log h_2/h_1 \colon X \setminus \overline{W} \to \R$ is a smooth strictly plurisubharmonic function.
The existence of a compact $\mathcal{F}$-invariant subset $M \subset X \setminus \overline{W}$ violates the maximum principle:
There is a point $p \in M$ where $\phi|_M$ takes its maximum since $M$ is compact.
Denote by $\mathcal{L}_p$ the leaf passing through $p$. 
Then, $\phi|_{\mathcal{L}_p}$, a smooth plurisubharmonic function on $\mathcal{L}_p$,
takes its maximum at $p$. 
Hence $\phi|_{\mathcal{L}_p}$ must be constant, but then cannot be a strictly plurisubharmonic function. 
This is a contradiction and completes the proof of the Main Theorem.
\end{proof}

\end{document}